\documentclass[12pt,a4paper]{amsart}
\usepackage{amsfonts}
\usepackage{amsthm}
\usepackage{amsmath}
\usepackage{amssymb}
\usepackage{amscd}
\usepackage[latin2]{inputenc}
\usepackage{t1enc}
\usepackage[mathscr]{eucal}
\usepackage{indentfirst}
\usepackage{graphicx}
\usepackage{graphics}
\numberwithin{equation}{section}
\usepackage[margin=2.9cm]{geometry}
\usepackage{epstopdf} 
\usepackage{xcolor}
\usepackage{comment}
\usepackage{enumitem}

\theoremstyle{plain}
\newtheorem{theorem}{Theorem}[section]
\newtheorem{lemma}[theorem]{Lemma}

\newtheorem{proposition}[theorem]{Proposition}

 \theoremstyle{definition}

\newtheorem{Rem}[theorem]{Remark}
\newtheorem{?}[theorem]{Problem}

\newcommand{\R}{\mathbb{R}}

\newcommand{\D}{\mathbb{D}}

\begin{document}

\title{The Narrow Capture Problem on General Riemannian Surfaces}

\author[M. Nursultanov]{Medet Nursultanov}
\address {Department of Mathematics and Statistics, University of Helsinki}
\email{medet.nursultanov@gmail.com}

\author[W. Trad] {William Trad}
\address {School of Mathematics and Statistics, University of Sydney}
\email{w.trad@maths.usyd.edu.au}

\author[J. C. Tzou] {Justin Tzou}
\address {School of Mathematical and Physical Sciences, Macquarie University}
\email{tzou.justin@gmail.com}

\author[L. Tzou] {Leo Tzou}
\address {University of Amsterdam}
\email{Leo.tzou@gmail.com}




\begin{abstract}
	In this article, we study the narrow capture problem on a Riemannian 2-manifold. This involves the derivation of the mean first passage (sojourn) time of a surface-bound ion modelled as a Brownian particle. We use a layer potential argument in conjunction with microlocal analysis in order to derive the leading order singularity as well as the $O(1)$ term of the mean first passage time and the associated spatial average. 

\end{abstract}
\maketitle
\section{Introduction}
We consider a Brownian particle bound to a surface which contains a small trap denoted $\Gamma_\varepsilon$. The narrow capture problem deals with the time required for such a particle to first encounter the trap. This time is called the first-passage time and is denoted $\tau_{\Gamma_{\varepsilon}}$. Starting from an initial location $x$ on the surface, the \emph{expected time} that a particle will wander before being captured by the trap is called the mean first-passage time and is denoted $u_\varepsilon(x)$. 

The narrow capture problem along with the closely related narrow escape problem (where the traps are small windows on the otherwise reflective boundary of the search domain) have been used as simple, prototypical models for various processes involving diffusive search such as a diffusing ion inside a cell escaping through an ion channel on the cell membrane (see \cite{pillay2010asymptotic, holcman2004escape, schuss2007narrow, benichou2008narrow, bressloff2013stochastic} and references therein). While we highlight some results below, we refer the reader to \cite{holcman2015stochastic, bressloff2014stochastic, HS2014} for a more comprehensive description of results along with their applications to cellular and molecular biology.

On a flat, bounded domain $\Omega$, it was shown in \cite{holcman2004escape} and \cite{singer2006narrow} that the mean escape time had the leading order expansion $u_\varepsilon \sim -|\Omega|\log\varepsilon + O(1)$ as $\varepsilon \to 0$. 
Within \cite{Singer2006}, similar expansions were derived for the sphere and annulus with absorbing windows near singular points. More precisely, it was found that in geometric configurations where the absorbing window was near a corner, the mean escape time had an expansion of the form $u_\varepsilon \sim \frac{|\Omega|}{\alpha}\log\varepsilon + O(1)$ where $\alpha$ denotes the angle of the corner. Furthermore, it was found that when the absorbing window was near a cusp point, the mean escape time had an expansion given by $u_\varepsilon \sim \frac{|\Omega|}{\varepsilon}+O(1)$. It is emphasized that in the above works, the leading order term was determined and the extent to which the remainder terms were understood was $O(1)$. 

In \cite{pillay2010asymptotic}, a matched asymptotic method was employed to determine the $O(1)$ term in the expansion in terms of a certain Green's function that encoded information on the geometry of $\Omega$, the locations of the traps, and the initial position $x$. This method, developed in \cite{ward1993strong}, effectively summed all logarithmic correction terms in the expansion of $u_\varepsilon$, with the resulting error term being transcendentally small in $\varepsilon$. In \cite{bressloff2015escape}, a more detailed model was considered in which the windows were allowed to open and close stochastically, more closely mimicking the behavior of cell ion channels. Other quantities of interest aside from the mean first passage time include the variance of the first passage time \cite{KTCW2015} as well as the so-called extreme first passage time, the minimum search time achieved by a large group of searchers \cite{madrid2020competition}. 

In three dimensions, the narrow escape time from a Euclidean, bounded domain $\Omega$ with one circular trap of radius $\varepsilon$ on its boundary was shown \cite{singer2006narrow, singer2008narrow} to have the leading order expansion $u_\varepsilon \sim |\Omega|(4\varepsilon)^{-1}(1-\varepsilon \pi^{-1} H\log\varepsilon ) + O(1)$, where $H$ is the mean curvature of $\partial \Omega$ at the center of the trap. In \cite{cheviakov2010asymptotic}, a matched asymptotic analysis similar to that employed in \cite{pillay2010asymptotic} was used to compute the $O(1)$ correction term in the expansion for the special case in which $\Omega$ was the unit ball. In \cite{NTT2021}, the calculation of this correction term was generalized using geometric microlocal methods to Riemannian 3-manifolds with smooth boundary containing either a circular or elliptical trap.

The narrow capture problem also has wide applications in cellular biology \cite{CW2011,CSW2009}. For example, a diffusing molecule must arrive at a localized signaling region within a cell or on its surface before a signaling cascade can be initiated. In another example, a T cell may diffuse in search of an antigen-presenting cell to trigger an immune response.  In this latter example, determining the duration of this search  is relevant to understanding immune response time \cite{CSW2009, delgado2015conditional}.

The matched asymptotic methods of \cite{ward1993strong} that were used for the aforementioned narrow escape problems have also been successfully applied to narrow capture problems in Euclidean metrics. In earlier works, the closely related problem of computing the fundamental Neumann eigenvalue $\lambda_0$ for the Laplacian in Euclidean two- and three-dimensional domains with small traps was considered in \cite{cheviakov2011optimizing, kolokolnikov2005optimizing}. The spatial average of the mean first passage time was shown to be proportional to $1/\lambda_0$ (see \cite{kolokolnikov2005optimizing, ward1993strong}), and in \cite{RC2019}, a numerical algorithm was employed to optimize this quantity with respect to configurations of traps located in the domain.

Extensions of these works include computing the full probability distribution (i.e., all moments as opposed to just the first)  of the first passage time \cite{lindsay2016hybrid, B20212d, B20213d, bressloff2022narrow}, stochastic resetting \cite{B20212d, B20213d, bressloff2022narrow}, moving traps \cite{tzou2015mean, lindsay2017optimization, iyaniwura2021simulation}, partially absorbing traps \cite{LBW2017, bressloff2022narrow}, traps grouped in clusters \cite{KTCW2015, iyaniwura2021asymptotic}, and the effect of advection \cite{KTCW2015, NTT2021force}.

In a non-Euclidean  geometry, \cite{CSW2009} considered the mean first passage time of a Brownian particle on a sphere containing small absorbing traps. Explicit results were obtained through employment of the aforementioned matched asymptotic method along with the known analytic formula for the Neumann Green's function for the Laplacian on sphere. The spherical geometry considered was meant to approximate the geometry of a cell with receptor clusters on its surface awaiting the arrival of surface-bound signaling molecules. A more detailed model of a cell, however, would be non-spherical. In fact, a cell's geometry  can be crucial to the manner in which it serves its function \cite{fernandes2019cell}. It is with this motivation that we develop here a rigorous mathematical framework for narrow capture problems posed on non-Euclidean and non-spherical geometries.

We now mathematically formulate the narrow capture problem. Let $(M,g)$ be a compact, connected, orientable, Riemannian surface with smooth boundary, $\partial M$. First, we assume the boundary to be empty and calculate the associated mean sojourn time as well as its spatial average. When the manifold is of non-empty boundary, we assume, without loss of generality that $M$ is a connected open subset of a compact orientable, Riemannian manifold $(\tilde{M}, g)$ without boundary. Let $(X_t, \mathbb{P}_t)$ be the Brownian motion on $M$ generated by the (negative) Laplace-Beltrami operator $\Delta_g=-d^*d$. We use $\Gamma_\varepsilon\subset M$ to denote a trap with radius $\varepsilon>0$ and we denote by $\tau_{\Gamma_{\varepsilon}}$ the first time the Brownian motion $X_t$ hits $\Gamma_{\varepsilon}$, that is
$$\tau_{\Gamma_{\varepsilon}} := \inf \{ t\geq 0: X_t \in \Gamma_{\varepsilon}\}.$$
Within the narrow capture problem we wish to derive an asymptotic as $\varepsilon \to 0$ for the mean first-passage (sojourn) time which is defined as the expected value: $$u_\varepsilon(x) = \mathbb{E}[\tau_{\Gamma_{\varepsilon}} | X_0 = x].$$ 
An associated quantity of interest is the spatial average of the mean first-passage time:
$$|M_\varepsilon|^{-1} \int_{M_\varepsilon} \mathbb{E}[\tau_{\Gamma_{\varepsilon, a}} | X_0 = x] dvol_g(x),$$
where $M_\varepsilon:= M\setminus \Gamma_\varepsilon$ and $|M_\varepsilon|$ is the Riemannian volume of $M_\varepsilon$ with respect to the metric $g$.

Many works have been devoted to this topic, especially in applied mathematics. In \cite{LL1986}, the mean first-passage time for diffusing particles on a surface of the sphere with one absorbing trap was considered. They obtained the asymptotic, up to the bounded term, for mean first passage time and its average. These results were generalised, in \cite{CSW2009}, for the case of several traps. In \cite{CW2011,LBW2017}, the three-dimensional version of this problem was studied. For domains in $\mathbb{R}^3$, they obtain the asymptotic formulas in terms of capacitance, by using the method of matched asymptotic expansions. We also mention works \cite{BEW2008,RC2019,TW2000,IWMW2021,B20212d,B20213d,SWF2007}, where the authors investigate related problems.

Despite the large number of works on this topic, there are still many questions regarding more general geometries. In this direction, the goal of this paper is to investigate the narrow capture problem for the Riemannian surface. Similar to \cite{AKK2012}, we use a layer potential method, however by adjoining this method with techniques originating from geometric microlocal analysis, we can extend the results, as well as the method to more general geometries, similar to the extension to broader classes of geometries for the narrow escape problem in \cite{NTT2021,NTT2021force}. As mentioned previously, we will consider empty and non-empty boundary cases. For the sake of conciseness, we will present the results and required Green's function for the $\partial M = \emptyset$ case here. For the associated \textit{Neumann Green's function} and results for the $\partial M\neq \emptyset$, see Section 5. 
\begin{equation}\label{G function}
\Delta_g E(x,y) = - \delta_y(x) + \frac{1}{|M|}, \quad E(x,y) = E(y,x), \quad \int_M E(x,y) dvol_g(y) = 0.
\end{equation}
It was already known, see for example \cite{NTT2021} and \cite{taylor2}, that near the diagonal the Green function satisfies 
\begin{equation}\label{dec}
	E(x,y) = -\frac{1}{2\pi} \log d_g(x,y) + P_{-4}(x,y),
\end{equation}
for some $P_{-4}(x,y) \in C^1(M\times M)$ which is infinitely smooth away from the set $\{(x,y)\in M\times M\mid x = y\}$. (In fact, in the language of pseudodifferential operators, we will see that $P_{-4}(x,y)$ is the Schwartz kernel of a pseudodifferential operator of degree $-4$.)

This expansion allows us to obtain the following asymptotic for the narrow capture of Brownian particles in a small trap:
\begin{theorem}\label{main}
	Let $(M,g)$ be a closed orientable Riemannian surface. Fix $x_0\in M$ and let $\Gamma_{\varepsilon}:=B_\varepsilon(x_0)$ be a geodesic ball centred at $x_0$ of geodesic radius $\varepsilon >0$.\\
	\\
	i) For each $x\notin B_\varepsilon(x_0)$, the first-passage time satisfies the  following asymptotic formula, as $\varepsilon\rightarrow 0$,
	\begin{equation*}
	\mathbb{E}[\tau_{\Gamma_{\varepsilon}} | X_0 = x] = -\frac{|M|}{2\pi}\log\varepsilon + |M|P_{-4}(x_0,x_0) - |M| E(x,x_0) + r_{\varepsilon}(x) + O(\varepsilon\log\varepsilon).
	\end{equation*}
	for some function $r_\varepsilon$ such that $\|r_\varepsilon\|_{C(K)} \leq C_{K} \varepsilon$ for any compact $K\subset M$ for which $K\cap \Gamma_\varepsilon = \emptyset$. The Green function $E(x,y)$ is given by \eqref{G function} and $P_{-4}(x_0,x_0)$ is the evaluation at $(x,y)=(x_0,x_0)$ of the $C^1(M\times M)$ function $P_{-4}(x,y)$ in \eqref{dec}.\\
	\\
	ii) Let $M_\varepsilon=M\setminus \Gamma_{\varepsilon}$, then the spatial average of the mean first-passage time satisfies the asymptotic formula, as $\varepsilon\rightarrow 0$,
	\begin{equation*}
	\frac{1}{|M_\epsilon|}\int_{M_{\varepsilon}}\mathbb{E}[\tau_{\Gamma_{\varepsilon}} | X_0 = y]dvol_g(y) = -\frac{|M|}{2\pi}\log\varepsilon +|M|P_{-4}(x_0,x_0) + O(\varepsilon\log\varepsilon).
	\end{equation*}
\end{theorem}

In Section $5$ Theorem \ref{main2} we will prove a similar result for $\partial M\neq \emptyset$ with reflection boundary conditions for the Brownian motion. In this setting, we will use instead the Neumann Green's function $E(x,y)$. See Section \ref{proof of thm with bry} for details.

This paper is structured in the following manner. In Section \ref{Prelim}, we introduce some notation and the geometric framework with which we will be operating. Section \ref{Green function} deals with investigating the singular structure of the Green's function on a Riemannian surface without boundary. In Section \ref{proof of thm without bry}, We make use of the derived Greens function to prove Theorem \ref{main}. In Section \ref{proof of thm with bry}, we consider the analogous problem in the setting of a manifold with boundary and impose a reflecting boundary condition for our Brownian motion. The result will be stated in Theorem \ref{main2}.

\section{Preliminaries}\label{Prelim}
Throughout this paper, $(M,g)$ be a compact connected orientable Riemannian surface with smooth boundary, $\partial M$ which could be empty.The corresponding volume form and geodesic distance are denoted by $dvol_g$ and $d_g(\cdot,\cdot)$, respectively.  By $|M|$ we denote the volume of $M$.

For fixed $x\in M$, we will denote by $B_\rho(x)$ the geodesic ball of radius $\rho>0$ centred at $x$. In what follows $\rho$ will always be smaller than the injectivity radius of $( M, g)$ and the distance from $x$ to $\partial M$. We let $\D_\rho$ be the Euclidean ball in $\R^2$ of radius $\rho$ centred at the origin. 

In this work, we will often use the geodesic coordinates constructed as follows. For fixed $x_0\in M$ and orthonormal tangent vectors $E_1, E_2 \in T_{x_0}M$, write $t=(t_1,t_2) \in \mathbb{D}_\rho$ and define 
\begin{equation}\label{geod coord}
x(t;x_0):= \exp_{x_0}(t_1E_1 + t_2E_2)
\end{equation}
where $\exp_{x_0} (V)$ denotes the time $1$ map of $g$-geodesics with initial point $x_0$ and initial velocity $V\in T_{x_0} M$. The coordinate $t\in \D_\rho \mapsto x(t; x_0)$ is then an $g$-geodesic coordinate system for a neighbourhood of $x_0$ on $M$.

We will also use the re-scaled version of this coordinate system. For $\varepsilon >0$ sufficiently small we define the (re-scaled) $g$-geodesic coordinate by the following map
\begin{equation}\label{res coord}
x^{\varepsilon}(\cdot ; x_0) : t = (t_1, t_2) \in \D \mapsto x(\varepsilon t; x_0) \in B_\varepsilon(x_0),
\end{equation}
where $ \mathbb{D}$ is the unit disk in $\R^2$.

In the subsequent sections, we denote the centre of the "trap" by $x_0\in M$, which will be considered as fixed. We will use the following notations. We set $\Gamma_\varepsilon := B_{\varepsilon}(x_0)$
\begin{equation*}
M_{\varepsilon}= M_\varepsilon(x_0) := M \setminus \Gamma_\varepsilon
\end{equation*}
and denote by $h=h(\varepsilon,x_0)$ the metric on $\partial M_\varepsilon$, induced by the trivial embedding of $\partial M_\varepsilon$ into $M_\varepsilon$. The corresponding volume form is denoted by $dvol_h$. Further, we set $z\in \partial M_\varepsilon \mapsto \nu_z$ to be an outward pointing normal for $M_\varepsilon$. Finally, we let $|M_\varepsilon|$, $|\partial \Gamma_\varepsilon|$ be the volumes of $M_\varepsilon$ and $\partial \Gamma_\varepsilon$ with respect to $g$ and $h$.

\section{Green's function}\label{Green function}
 Within this section we assume that $\partial M$ is empty and we consider the Green function on $M$, which is the fundamental solution to the Laplace equation:
\begin{equation}
\label{no boundary green}
\Delta_g E(x,y) = - \delta_y(x) + \frac{1}{|M|}, \quad E(x,y) = E(y,x), \quad \int_M E(x,y) dvol_g(y) = 0.
\end{equation}

For a fixed $x_0\in M$ and set $\Gamma_\varepsilon = B_{\varepsilon}(x_0)$ we consider the following function
\begin{equation*}
	I_{\varepsilon}(x_0,x) := \int_{\Gamma_\varepsilon} E(x,y) dvol_g(y).
\end{equation*}
for $x\in M_\epsilon$. We will need to know about the singular behaviour of $\partial_{\nu_x} E(\cdot,\cdot)$, $I_\varepsilon(x_0,\cdot)$, and $\partial_{\nu_x}I_{\varepsilon}(x_0,\cdot)$ on $\partial \Gamma_\varepsilon$ as we approach neighbourhoods of the diagonal. To investigate these, we recall the singularity structure of $E(\cdot,\cdot)$:

\begin{proposition}\label{sing structure E}
	The Green function $E(x,y)$ and has the following singularity structure near the diagonal
	\begin{equation*}
	E(x,y) = -\frac{1}{2\pi} \log d_g(x,y) + P_{-4}(x,y),
	\end{equation*}
	where $P_{-4}(x,y)\in C^1(M\times M)$ is infinitely differentiable off the diagonal $\{x=y\}$.
\end{proposition}
We will prove proposition 3.1 in Section \ref{proof of lemma 3.1} as it involves the use of pseudodifferential operators. 

As the distance function plays a crucial role in the Green's function $E(x,y)$, it is useful to derive asymptotics for them in the appropriate coordinate systems:

\begin{lemma}
Let
$$d_g^*(s,t): = d_g(x(s,x_0),x(t,x_0)).$$
where $t  = (t_1, t_2) \in \D_\rho \mapsto x(t,x_0)$ is the coordinate system defined in \eqref{geod coord}. Then we have that
\begin{equation*}
	d_g^*(s,t) = |s - t| + |s - t| F\left(t,\frac{s - t}{|s - t|}, |s - t|\right)
	\end{equation*}
for some smooth function $F(t,\omega, r)\in C^\infty(\D_\rho \times S^1\times[0,2\rho])$  which is $O(t) + O(r)$.
\end{lemma}

\begin{proof}
By Lemma 4.8 of \cite{lefeuvretensor}, if $t\mapsto x(t,x_0)$ is any coordinate system, there exists a matrix $H_{j,k}(s,t)$ smooth in $(s,t)$ such that
\begin{eqnarray}
\label{dstar}
d_{g}^*(s,t)^2 = \sum_{j=1}^2H_{j,k}(s,t)(s_j - t_j) (s_k - t_k),
\end{eqnarray}
where $H_{j,k}(t,t)  = g_{j,k}(t)$ is the coordinate expression for the metric tensor $g$. Since the coordinate system \eqref{geod coord} is the geodesic coordinate system, we have that  $g_{j,k}(t) =\delta_{j,k} + O(|t|^2)$. So we get $H_{j,k}(t, t) = \delta_{j,k} + O(|t|^2)$. Taylor expand $H_{j,k}(s,t)$ around $s=t$ and insert the resulting expression into \eqref{dstar} we get
$$d_{g}^*(s,t) = |s-t| + |s-t| F\left(t, \frac{s-t}{ |s-t|},  |s-t|\right)$$
for some smooth function $F\in C^\infty(\D_\rho \times S^1\times[0,r_0])$ which is $O(t) + O(r)$.
\end{proof}

The following distance expression in the rescaled normal coordinates given by \eqref{res coord} was stated in Corollary 2.6 of \cite{NTT2021}:
\begin{lemma}
\label{d inverse} 
For the coordinates given by \eqref{res coord},
\begin{equation*}
		d^{-1}_g(x^\varepsilon(s,x_0), x^\varepsilon(t,x_0)) = \varepsilon^{-1}|t - s|^{-1} + \varepsilon |t - s|^{-1}A(\varepsilon, s,r,\omega)
	\end{equation*}
for some smooth function $A$ in the variables $(\varepsilon,s,r,\omega)\in [0,\varepsilon_0]\times\mathbb{D}\times\mathbb{R}\times S^1$, where $r=|t-s|$ and $\omega = \frac{t-s}{|t-s|}$.
\end{lemma}

In the next two lemmas, we investigate the properties of $I_{\varepsilon}(x_0,\cdot)$
\begin{lemma}\label{I}
	The following estimate holds
	\begin{equation}\label{sing of I}
		\sup_{x\in \partial \Gamma_\varepsilon} I_{\varepsilon}(x_0, x) = O(\varepsilon^2 \log\varepsilon), \qquad
		\text{as } \varepsilon \rightarrow 0.
	\end{equation}
\end{lemma}

\begin{proof}
	Due to Proposition \ref{sing structure E} it is sufficient to prove that
	\begin{equation*}
		\sup_{x\in \partial \Gamma_\varepsilon} \int_{\Gamma_\varepsilon} \log d_g(x,y) dvol_g(y) = O(\varepsilon^2 \log\varepsilon).
	\end{equation*}
	We consider $\varepsilon>0$ sufficiently small, so that $\log (10\varepsilon) < 0$. Then, for $x\in \partial \Gamma_\varepsilon$, 
	\begin{equation*}
	\left|\int_{\Gamma_\varepsilon} \log d_g(x,y) dvol_g(y)\right| =
	\left|\int_{B_{\varepsilon}(x_0)} \log d_g(x,y) dvol_g(y)\right| \leq \left|\int_{B_{2\varepsilon}(x)} \log d_g(x,y) dvol_g(y)\right|.
	\end{equation*}
	For $\varepsilon>0$ sufficiently small we can find $\rho>3\varepsilon$ which is smaller than the injectivity radius. We will use the coordinate system given by 
	\begin{equation*}
		\mathbb{D}_{\rho} \ni (s_1,s_2) \mapsto x(s_1,s_2;x_0),
	\end{equation*}
	defined in Section \ref{Prelim}. We recall that $s=(s_1,s_2)$ and $t=(t_1,t_2)$ and let
$$d_g^*(s,t): = d_g(x(s,x_0),x(t,x_0)).$$
Lemma 3.2 tells us that 
	\begin{equation*}
	d_g^*(s,t) = |s - t| + |s - t| F\left(t,\frac{s - t}{|s - t|}, |s - t|\right)
	\end{equation*}
	for some smooth function $F$ which is $O(|t|) + O(|s - t|)$. Therefore, for sufficiently small $\varepsilon >0$, we can choose $\rho>0$ small enough so that for all $s,t \in \D_\rho$, 
	\begin{equation*}
	\frac{1}{2}|s - t| \leq d_g^*(t,s) \leq 2|s - t|.
	\end{equation*}
	Furthermore, we choose $C > 0$ such that $\sqrt{\textrm{det}( g_{j,k}(s))} \leq C$ for $s\in \mathbb{D}_{\rho}$. Therefore, for $x = x(t,x_0)\in \partial \Gamma_\varepsilon$, we estimate
	\begin{align*}
		\left|\int_{B_{2\varepsilon}(x)} \log d_g(x,y) dvol_g(y)\right| 
		&\leq \left| C\int_{d_g^*(t,s)\leq 2\varepsilon} \log\left(2|s - t|\right) ds \right|\\
		&\leq \left| C\int_{|s - t|\leq 4\varepsilon} \log\left(2|s - t|\right) ds \right| = O(\varepsilon^2 \log\varepsilon).
	\end{align*}
\end{proof}

\begin{lemma}\label{partial I}
	The following estimate holds
	\begin{equation}\label{sing of partialI}
	\sup_{x\in \partial \Gamma_\varepsilon} \partial_{\nu_x} I(x_0,x) = O(\varepsilon),
	\qquad \text{as } \varepsilon \rightarrow 0.
	\end{equation}
\end{lemma}

\begin{proof}
	Let us use the coordinate system $x(t,x_0)$. Note that in these coordinates the volume form for $M$ is given by
	\begin{equation}
		dvol_g(y) = (1 + \varepsilon V_{\varepsilon}(s))ds_1\wedge ds_2
	\end{equation}
	for some smooth function $V_\varepsilon (s)$ whose derivatives of all orders are bounded uniformly in $\varepsilon$. We also note that in these coordinates, we have
	\begin{equation}\label{dist}
	d_g(x(t;x_0), x(s;x_0))^2 = \sum_{\alpha,\beta=1}^{2}G_{\alpha,\beta}(s,t)(s_\alpha - t_\alpha)(s_\beta - t_\beta)
	\end{equation}
	where $t=(t_1,t_2)$, $s=(s_1,s_2)$, and $G_{\alpha,\beta}(s,t)$ is a smooth function on $\mathbb{D}\times\mathbb{D}$ such that $G_{\alpha,\beta}(s,s) = \delta_{\alpha}^{\beta} + O(|s|^2)$ for $s$ near $0$. Then, by Proposition 2.8 in \cite{taylor2}, we know
	\begin{align*}
	E(x(t;x_0), x(s;x_0)) = &-\frac{1}{4\pi} \log \left(\sum_{\alpha,\beta=1}^{2}G_{\alpha,\beta}(s,t)(s_\alpha - t_\alpha)(s_\beta - t_\beta)\right)\\
	&+ q_2(s,s-t) + p_2(s,s-t) \log |s - t|  +R(s,t).
	\end{align*}
	Here $p_2(x,z)$ is a polynomial homogeneous of degree $2$ in $z$, with the coefficients that are bounded, together with their $x$-derivatives. A function $q_2(x,z)$ is smooth on $\mathbb{R}^2\setminus\{0\}$ and homogeneous of degree $2$ in $z$. Finally, $R\in C^2(\mathbb{R}^2_s \times \mathbb{R}^2_t)$.
	
	Let us use the polar coordinates
	\begin{equation*}
		t = (r\cos\theta, r\sin \theta) \qquad s = (r'\cos\theta', r'\sin \theta').
	\end{equation*}
	We note that $x(\{|t|= \varepsilon;x_0\}) = \partial \Gamma_\varepsilon$ and $r \mapsto ( r\cos\theta, r \sin\theta)$ is the parametrization of unit speed geodesic issued from the origin. Therefore, since $\partial_{\nu_x}$ is the inward normal of $\partial \Gamma_\varepsilon$ for $x = ( r\cos\theta, r \sin\theta)\in \partial \Gamma_\varepsilon$, it follows from Gauss Lemma that $\Phi_* \partial_{\nu_x} = -\partial_r \in T_{(r\cos\theta, r \sin\theta)}\mathbb{R}^2$. Therefore
	\begin{align*}
	\partial_{\nu_x}E(x, y) = \partial_r \Bigg[-\frac{1}{4\pi} \log &\left(\sum_{\alpha,\beta=1}^{2}G_{\alpha,\beta}(s,t)(s_\alpha - t_\alpha)(s_\beta - t_\beta)\right)\\
	&+ q_2(s,s-t) + p_2(s,s-t)\log |s - t|  +R(s,t)\Bigg].
	\end{align*}
	Therefore
	\begin{align*}
	\partial_x I(x_0,x) = &\int_{\mathbb{D}_\varepsilon}\partial_r \Bigg[-\frac{1}{4\pi} \log \left(\sum_{\alpha,\beta=1}^{2}G_{\alpha,\beta}(s,t)(s_\alpha - t_\alpha)(s_\beta - t_\beta)\right)\\
	&+ q_2(s,s-t) + p_2(s,s-t)\log |s - t|  +R(s,t)\Bigg](1 + \varepsilon V_\varepsilon(s))ds.
	\end{align*}
	From the properties of functions $q_2$, $p_2$, and $R$, mentioned above, it follows that
	\begin{equation*}
	\int_{\mathbb{D}_\varepsilon} \partial_r \left( q_2(s,s-t) + p_2(s,s-t)\log |s - t|  +R(s,t) \right)(1 + \varepsilon V_\varepsilon(s))ds = O(\varepsilon^2)
	\end{equation*}
	uniformly on $t$. Hence, we have
	\begin{align}\label{I in eucl}
		\nonumber\partial_x I(x_0,x) = &-\frac{1}{4\pi} \int_{\mathbb{D}_\varepsilon}\partial_r  \log \left(\sum_{\alpha,\beta=1}^{2}G_{\alpha,\beta}(s,t)(s_\alpha - t_\alpha)(s_\beta - t_\beta)\right)ds\\
		& -\frac{\varepsilon}{4\pi} \int_{\mathbb{D}_\varepsilon}\partial_r  \log \left(\sum_{\alpha,\beta=1}^{2}G_{\alpha,\beta}(s,t)(s_\alpha - t_\alpha)(s_\beta - t_\beta)\right) V_\varepsilon(s)ds + O(\varepsilon^2)
	\end{align}
	The first integral of the right-hand side is equal to 
	\begin{align}
	\nonumber \int_{0}^{2\pi}&\int_{0}^{\varepsilon} \frac{2\begin{bmatrix}
		\cos\theta & \sin\theta
		\end{bmatrix}
		G
		\begin{bmatrix}
		r\cos\theta - r'\cos\theta'\\
		r\sin\theta - r'\sin\theta'
		\end{bmatrix}}{\begin{bmatrix}
		r\cos\theta - r'\cos\theta' & r\sin\theta - r'\sin\theta'
		\end{bmatrix}
		G
		\begin{bmatrix}
		r\cos\theta - r'\cos\theta'\\
		r\sin\theta - r'\sin\theta'
		\end{bmatrix}}r'dr'd\theta'\\
	&+ \int_{0}^{2\pi}\int_{0}^{\varepsilon} \frac{\begin{bmatrix}
		r\cos\theta - r'\cos\theta' & r\sin\theta - r'\sin\theta'
		\end{bmatrix}
		\partial_rG
		\begin{bmatrix}
		r\cos\theta - r'\cos\theta'\\
		r\sin\theta - r'\sin\theta'
		\end{bmatrix}}{\begin{bmatrix}
		r\cos\theta - r'\cos\theta' & r\sin\theta - r'\sin\theta'
		\end{bmatrix}
		G
		\begin{bmatrix}
		r\cos\theta - r'\cos\theta'\\
		r\sin\theta - r'\sin\theta'
		\end{bmatrix}}r'dr'd\theta',
	\end{align}
	where $G=G(r,\theta,r',\theta')$ is two by two matrix with entries $\{G_{\alpha,\beta}(s,t)\}$ with $t = t(r,\theta)$ and $s = s(r',\theta')$. Since $x\in \partial B_\varepsilon(x_0)$, we take $r=\varepsilon$. Then, if we change the variable $r' \mapsto \varepsilon r'$, the last expression becomes
	\begin{align*}
	\varepsilon\int_{0}^{2\pi}&\int_{0}^{1} \frac{2\begin{bmatrix}
		\cos\theta & \sin\theta
		\end{bmatrix}
		G
		\begin{bmatrix}
		\cos\theta - r'\cos\theta'\\
		\sin\theta - r'\sin\theta'
		\end{bmatrix}}{\begin{bmatrix}
		\cos\theta - r'\cos\theta' & \sin\theta - r'\sin\theta'
		\end{bmatrix}
		G
		\begin{bmatrix}
		\cos\theta - r'\cos\theta'\\
		\sin\theta - r'\sin\theta'
		\end{bmatrix}}r'dr'd\theta'\\
	&+ \varepsilon^2\int_{0}^{2\pi}\int_{0}^{1} \frac{\begin{bmatrix}
		\cos\theta - r'\cos\theta' & \sin\theta - r'\sin\theta'
		\end{bmatrix}
		\partial_rG
		\begin{bmatrix}
		\cos\theta - r'\cos\theta'\\
		\sin\theta - r'\sin\theta'
		\end{bmatrix}}{\begin{bmatrix}
		\cos\theta - r'\cos\theta' & \sin\theta - r'\sin\theta'
		\end{bmatrix}
		G
		\begin{bmatrix}
		\cos\theta - r'\cos\theta'\\
		\sin\theta - r'\sin\theta'
		\end{bmatrix}}r'dr'd\theta'.
	\end{align*}
	Note that we have integrable singularity at point $(r',\theta') = (1,\theta)$ and integrals are bounded uniformly on $\theta$. Therefore, the last expression is $O(\varepsilon)$ as $\varepsilon\rightarrow 0$ uniformly on $\theta$. Since $V_\varepsilon(s)$ is bounded uniformly on $\varepsilon$, the second term of \eqref{I in eucl} is $O(\varepsilon^3)$.
\end{proof}

Next, we obtain the singularity structure of $\partial_{\nu_x} E(\cdot,\cdot)$ in a neighbourhood of $x_0$:

\begin{lemma}\label{sing str of E}
	Let $B_{\varepsilon}(x_0)$ be the geodesic ball with radius $\varepsilon$ centred at $x_0$. Then
	\begin{equation*}
	\partial_{\nu_x} E(x,y)\left.\right|_{x,y\in \partial \Gamma_\varepsilon} = \frac{1}{4\pi\varepsilon} + Q_{\varepsilon}(x,y),
	\end{equation*}
	for some function $Q_\varepsilon$ such that 
	\begin{equation*}
	\sup_{x\in \partial \Gamma_\varepsilon} \int_{\partial \Gamma_\varepsilon}  Q_{\varepsilon}(x,y) dvol_h(y) = O(\varepsilon).
	\end{equation*}
\end{lemma}

\begin{proof}
	We begin as in the prove of Lemma \ref{partial I}. We repeat all steps until we derive
	\begin{align}\label{partial x E 2}
	\nonumber\partial_{\nu_x}E(x, y) = \partial_r \Bigg[-\frac{1}{4\pi} \log &\left(\sum_{\alpha,\beta=1}^{2}G_{\alpha,\beta}(s,t)(s_\alpha - t_\alpha)(s_\beta - t_\beta)\right)\\
	&+ q_2(s,s-t) + p_2(s,s-t) \log |s - t| +R(s,t)\Bigg].
	\end{align}
	We recall that $p_2(x,z)$ is a polynomial homogeneous of degree $2$ in $z$, with the coefficients that are bounded, together with their $x-$derivatives. A function $q_2(x,z)$ is smooth on $\mathbb{R}^2\setminus\{0\}$ and homogeneous of degree $2$ in $z$. Finally, $R\in C^2(\mathbb{R}^2_s \times \mathbb{R}^2_t)$. These conditions imply that
	\begin{equation*}
		 \int_0^{2\pi}\partial_r \left[q_2(s,s-t) + p_2(s,s-t) \log |s - t| +R(s,t)\right] \left.\right|_{r=r'=\varepsilon}d\theta' = O(1)
	\end{equation*}
	as $\varepsilon\rightarrow 0$ uniformly on $\theta$. 
	
	Next, we investigate the first term of the right-hand side of \eqref{partial x E 2}, which can be written as follows
	\begin{multline*}
	\frac{1}{4\pi} \frac{2\begin{bmatrix}
		\cos\theta & \sin\theta
		\end{bmatrix}
		G
		\begin{bmatrix}
		r\cos\theta - r'\cos\theta'\\
		r\sin\theta - r'\sin\theta'
		\end{bmatrix}}{\begin{bmatrix}
		r\cos\theta - r'\cos\theta' & r\sin\theta - r'\sin\theta'
		\end{bmatrix}
		G
		\begin{bmatrix}
		r\cos\theta - r'\cos\theta'\\
		r\sin\theta - r'\sin\theta'
		\end{bmatrix}}\\
	+\frac{1}{4\pi} \frac{\begin{bmatrix}
		r\cos\theta - r'\cos\theta' & r\sin\theta - r'\sin\theta'
		\end{bmatrix}
		\partial_rG
		\begin{bmatrix}
		r\cos\theta - r'\cos\theta'\\
		r\sin\theta - r'\sin\theta'
		\end{bmatrix}}{\begin{bmatrix}
		r\cos\theta - r'\cos\theta' & r\sin\theta - r'\sin\theta'
		\end{bmatrix}
		G
		\begin{bmatrix}
		r\cos\theta - r'\cos\theta'\\
		r\sin\theta - r'\sin\theta'
		\end{bmatrix}}.
	\end{multline*}
	Since $x$, $y\in \partial \Gamma_\varepsilon$, we take $r=r'=\varepsilon$, so that the last expression becomes
	\begin{multline*}
	\frac{1}{2\pi\varepsilon} \frac{\begin{bmatrix}
		\cos\theta & \sin\theta
		\end{bmatrix}
		G
		\begin{bmatrix}
		\cos\theta - \cos\theta'\\
		\sin\theta - \sin\theta'
		\end{bmatrix}}{\begin{bmatrix}
		\cos\theta - \cos\theta' & \sin\theta - \sin\theta'
		\end{bmatrix}
		G
		\begin{bmatrix}
		\cos\theta - \cos\theta'\\
		\sin\theta - \sin\theta'
		\end{bmatrix}}\\
	+\frac{1}{4\pi} \frac{\begin{bmatrix}
		\cos\theta - \cos\theta' & \sin\theta - \sin\theta'
		\end{bmatrix}
		\partial_rG
		\begin{bmatrix}
		\cos\theta - \cos\theta'\\
		\sin\theta - \sin\theta'
		\end{bmatrix}}{\begin{bmatrix}
		\cos\theta - \cos\theta' & \sin\theta - \sin\theta'
		\end{bmatrix}
		G
		\begin{bmatrix}
		\cos\theta - \cos\theta'\\
		\sin\theta - \sin\theta'
		\end{bmatrix}}.
	\end{multline*}
	Note that the last term belongs to $L^{\infty}(S^1_{\theta} \times S^1_{\theta'})$ uniformly in $\varepsilon$. Therefore, it remains to show that 
	\begin{equation}\label{it remains to show}
		\frac{1}{2\pi\varepsilon} \frac{\begin{bmatrix}
			\cos\theta & \sin\theta
			\end{bmatrix}
			G
			\begin{bmatrix}
			\cos\theta - \cos\theta'\\
			\sin\theta - \sin\theta'
			\end{bmatrix}}{\begin{bmatrix}
			\cos\theta - \cos\theta' & \sin\theta - \sin\theta'
			\end{bmatrix}
			G
			\begin{bmatrix}
			\cos\theta - \cos\theta'\\
			\sin\theta - \sin\theta'
			\end{bmatrix}}
		= \frac{1}{4\pi\varepsilon} + L_\varepsilon(\theta, \theta')
	\end{equation}
	for some function $L_{\varepsilon}$ such that 
	\begin{equation*}
		\int_{0}^{2\pi}L_\varepsilon(\theta, \theta')d\theta' = O(1), \qquad \text{as } \varepsilon\rightarrow 0,
	\end{equation*}
	uniformly on $\theta$. Let $J(\theta,\theta')$ be the left-hand side of \eqref{it remains to show}. We denote $J_{1} : = J \chi_{|\theta - \theta'|<\varepsilon}$ and $J_{2} : = J \chi_{|\theta - \theta'|>\varepsilon}$, where $\chi$ is an indicator function of the corresponding set.
	
	To investigate $J_1$, we will use Taylor expansion for its numerator and denominator at $\theta'=\theta$. We recall that $G_{j,k}(x,x) = g_{j,k}(x)$, for $r=r'=\varepsilon$, we get
	\begin{equation}\label{G}
	G = g +\varepsilon R_{\varepsilon}(\theta,\theta') (\theta - \theta'),
	\end{equation}
	where $R_\varepsilon$ is two by two matrix with $C^{\infty}(S^1_{\theta} \times S^1_{\theta'})$ entries and $g = \{g_{j,k}(\varepsilon\cos \theta, \varepsilon\sin \theta)\}_{k,j}^2$. Furthermore, we express $g$ in the following way
	\begin{equation}\label{g}
	g = I + \Gamma(\varepsilon,\theta),
	\end{equation}
	where $I$ is two by two identity matrix and $\Gamma$ is two by two matrix with interiors $O(\varepsilon^2)$. Therefore, by applying Taylor expansion at $\theta = \theta'$, we obtain
	\begin{multline*}
	\begin{bmatrix}
	\cos\theta & \sin\theta
	\end{bmatrix}
	G
	\begin{bmatrix}
	\cos\theta - \cos\theta'\\
	\sin\theta - \sin\theta'
	\end{bmatrix}
	=
	\begin{bmatrix}
	\cos\theta & \sin\theta
	\end{bmatrix}
	g
	\begin{bmatrix}
	- \sin\theta\\
	\cos\theta
	\end{bmatrix}\\
	+\begin{bmatrix}
	\cos\theta & \sin\theta
	\end{bmatrix}
	g
	\begin{bmatrix}
	K_1(\theta,\theta')\\
	K_2(\theta,\theta')
	\end{bmatrix}(\theta-\theta')^2+ O(\varepsilon) O( |\theta-\theta'|^2),
	\end{multline*}
	for some $K=(K_1,K_2) \in L^{\infty}\left(S^1\times S^1\right)^2$. Note that the normal vector on $\{|t| = \varepsilon\}$ is given by $\cos\theta\partial_{t_1} + \sin\theta\partial_{t_2}$ at the point $(\varepsilon\cos\theta, \varepsilon\sin\theta)$, while the tangent is given by $-\sin\theta\partial_{t_1} + \cos\theta\partial_{t_2}$. Therefore the first term of the right-hand side of the last equation is zero, so that
	\begin{align*}
	\begin{bmatrix}
	\cos\theta & \sin\theta
	\end{bmatrix}
	G
	\begin{bmatrix}
	\cos\theta - \cos\theta'\\
	\sin\theta - \sin\theta'
	\end{bmatrix}
	=\begin{bmatrix}
	\cos\theta & \sin\theta
	\end{bmatrix}
	g
	\begin{bmatrix}
	K_1\\
	K_2
	\end{bmatrix}(\theta-\theta')^2 + O(\varepsilon)O( |\theta-\theta'|^2).
	\end{align*}
	Similarly, by using \eqref{G} and \eqref{g}, we show that
	\begin{align*}
	\begin{bmatrix}
	\cos\theta - \cos\theta' & \sin\theta - \sin\theta'
	\end{bmatrix}
	G
	\begin{bmatrix}
	\cos\theta - \cos\theta'\\
	\sin\theta - \sin\theta'
	\end{bmatrix}
	=2(\theta - \theta')^2 + O(\varepsilon)O( |\theta-\theta'|^2).
	\end{align*}
	The last two estimates imply that
	\begin{equation*}
	\int_{0}^{2\pi} J_1d\theta' = \frac{1}{2\pi\varepsilon} \int_{0}^{2\pi} 
	\frac{\begin{bmatrix}
		\cos\theta & \sin\theta
		\end{bmatrix}
		G
		\begin{bmatrix}
		K_1\\
		K_2
		\end{bmatrix} + O(\varepsilon)}
	{2 + O(\varepsilon)}\chi_{|\theta - \theta'|<\varepsilon}(\theta') d\theta' = O(1)
	\end{equation*}
	as $\varepsilon\rightarrow0$ uniformly in $\theta$.
	
	Next, we will investigate $J_2$. From \eqref{G} and \eqref{g}, it follows that 
	\begin{equation*}
		\begin{bmatrix}
		\cos\theta & \sin\theta
		\end{bmatrix}
		G
		\begin{bmatrix}
		\cos\theta - \cos\theta'\\
		\sin\theta - \sin\theta'
		\end{bmatrix} = 1 - \cos(\theta - \theta') + O(\varepsilon^2)O(|\theta - \theta'|) + O(\varepsilon)O(|\theta - \theta'|^2).
	\end{equation*}
	In the region $\{|\theta - \theta'|>\varepsilon\}$, we can rewrite this
	\begin{equation*}
	\begin{bmatrix}
	\cos\theta & \sin\theta
	\end{bmatrix}
	G
	\begin{bmatrix}
	\cos\theta - \cos\theta'\\
	\sin\theta - \sin\theta'
	\end{bmatrix} = 1 - \cos(\theta - \theta') + O(\varepsilon)O(|\theta - \theta'|^2).
	\end{equation*}
	Similarly,
	\begin{equation*}
		\begin{bmatrix}
		\cos\theta - \cos\theta' & \sin\theta - \sin\theta'
		\end{bmatrix}
		G
		\begin{bmatrix}
		\cos\theta - \cos\theta'\\
		\sin\theta - \sin\theta'
		\end{bmatrix}
		= 2 - 2\cos(\theta - \theta') + O(\varepsilon)O(|\theta - \theta'|^3).
	\end{equation*}
	Therefore, we have
	\begin{align*}
		J_2 &= \frac{1}{2\pi\varepsilon}\left(\frac{1}{2} + \frac{O(\varepsilon)O(|\theta - \theta'|^2)}{2 - 2\cos(\theta - \theta') + O(\varepsilon)O(|\theta - \theta'|^3)}\right)\\
		&= \frac{1}{4\pi\varepsilon} + \frac{O(1)}{\frac{2 - 2\cos(\theta - \theta')}{(\theta - \theta')^2} + O(\varepsilon)O(|\theta - \theta'|)}
	\end{align*}
	Since $(2 - 2\cos(\theta - \theta')) (\theta - \theta')^{-2}$ is a positive and continuous function of $\theta'\in [0,2\pi]$, we conclude that
	\begin{equation*}
		\int_{0}^{2\pi}\left(J_2 - \frac{1}{4\pi\varepsilon}\right)d\theta' = O(1) \qquad \text{as } \varepsilon\rightarrow 0,
	\end{equation*}
	uniformly in $\theta$. The Lemma is proved.
\end{proof}

\section{Narrow capture problem on the surface without boundary}\label{proof of thm without bry}
In this section, we prove Theorem \ref{main}. We start by recalling the formulation of the problem. Let $(X_t, \mathbb{P}_x)$ be the Brownian motion on a boundaryless manifold $M$ starting at $x$, generated by $\Delta_g$. For $x_0\in M$ and  $\varepsilon>0$, let $\Gamma_{\varepsilon} =  B_\varepsilon(x_0)$ be a small geodesic ball centred at fixed point $x_0\in M$.
Denote by $\tau_{\Gamma_{\varepsilon}}$ the first time the Brownian motion $X_t$ hits $\Gamma_{\varepsilon}$, that is
\begin{equation*}
\tau_{\Gamma_{\varepsilon}} := \inf \{ t\geq 0: X_t \in \Gamma_{\varepsilon}\}.
\end{equation*}
We aim to investigate the mean first-passage time and its average:
\begin{equation*}
\mathbb{E}[\tau_{\Gamma_{\varepsilon}} | X_0 = x], 
\qquad |M_\varepsilon|^{-1} \int_{M_\varepsilon} \mathbb{E}[\tau_{\Gamma_{\varepsilon}} | X_0 = x] dvol_g(x).
\end{equation*}
where $M_{\varepsilon}:=M\setminus \Gamma_\varepsilon$. Namely, we want to derive asymptotic expansion for these quantities as $\varepsilon\to 0$. It is known that $\mathbb{E}[\tau_{\Gamma_{\varepsilon}} | X_0 = x]$ satisfies the following boundary value problem, see for instance Appendix A in \cite{NTT2021},
\begin{equation}\label{bvp}
\begin{cases}
\Delta_g u_\varepsilon = - 1 & \text{on } M_{\varepsilon};\\
u_\varepsilon = 0 & \text{on } \partial M_{\varepsilon}=\partial \Gamma_\varepsilon,
\end{cases}
\end{equation}
which gives the compatibility condition
\begin{equation}\label{comp cond}
\int_{\partial \Gamma_\varepsilon}\partial_{\nu}u_{\varepsilon}(y)dvol_h(y) = - |M_{\varepsilon}|.
\end{equation}

To prove Theorem \ref{main}, we will need the following auxiliary result.
\begin{proposition}\label{partial_u}
	Let $u_{\varepsilon}$ be the solution of \eqref{bvp}, then
	\begin{equation*}
	\partial_{\nu}\left.u_{\varepsilon}\right|_{\partial \Gamma_\varepsilon} = -\frac{|M_{\varepsilon}|}{2\pi\varepsilon} + W_\varepsilon.
	\end{equation*}
	for some $W_\varepsilon\in O_{L^{\infty}(\partial \Gamma_\varepsilon)}(1)$ as $\varepsilon\rightarrow 0$.
\end{proposition}

\begin{proof}
	By using Green's identity, we obtain
	\begin{equation}\label{u_partialu_eq}
	\frac{1}{|M|}\int_{M_{\varepsilon}}u_{\varepsilon}(y)dvol_g(y) - u_{\varepsilon}(x) + \int_{\partial \Gamma_\varepsilon}E(x,y)\partial_{\nu}u_{\varepsilon}(y)dvol_h(y) = I_{\varepsilon}(x_0,x).
	\end{equation}
	We take $\partial_{\nu_x}$ and restrict to $\partial \Gamma_\varepsilon$
	\begin{equation*}
	- \partial_{\nu_x}u_{\varepsilon}(x) + \partial_{\nu_x} \int_{\partial \Gamma_\varepsilon}E(x,y)\partial_{\nu}u_{\varepsilon}(y)dvol_h(y) = \partial_{\nu_x}I_{\varepsilon}(x_0,x),
	\end{equation*}
	and hence, by Lemma \ref{partial I}, we derive
	\begin{equation*}
	- \partial_{\nu_x}u_{\varepsilon}(x) + \partial_{\nu_x} \int_{\partial \Gamma_\varepsilon}E(x,y)\partial_{\nu}u_{\varepsilon}(y)dvol_h(y) = O_{L^{\infty}(\partial \Gamma_\varepsilon)}(\varepsilon).
	\end{equation*}
	By Proposition 11.3 of \cite{taylor2},
	\begin{equation*}
	- \frac{1}{2} \partial_{\nu}u_{\varepsilon}(x) +  \int_{\partial \Gamma_\varepsilon}\partial_{\nu_x}E(x,y)\partial_{\nu}u_{\varepsilon}(y)dvol_h(y) = O_{L^{\infty}(\partial \Gamma_\varepsilon)}(\varepsilon).
	\end{equation*}
	Therefore, from Lemma \ref{sing str of E}, it follows
	\begin{equation*}
	\frac{1}{2} \partial_{\nu}u_{\varepsilon}(x) =  \frac{1}{4\pi\varepsilon}\int_{\partial \Gamma_\varepsilon}\partial_{\nu}u_{\varepsilon}(y)dvol_h(y)  + \int_{\partial \Gamma_\varepsilon}Q_{\varepsilon}(x,y)\partial_{\nu}u_{\varepsilon}(y)dvol_h(y)+ O_{L^{\infty}(\partial \Gamma_\varepsilon)}(\varepsilon).
	\end{equation*}
	Hence, the compatibility condition \eqref{comp cond} gives 
	\begin{equation}
\label{epsilonpartialnu}
	\varepsilon \partial_{\nu}u_{\varepsilon}(x) =  - \frac{|M_{\varepsilon}|}{2\pi}   + 2\varepsilon\int_{\partial \Gamma_\varepsilon}Q_{\varepsilon}(x,y)\partial_{\nu}u_{\varepsilon}(y) dvol_h(y) + O_{L^{\infty}(\partial \Gamma_\varepsilon)}(\varepsilon^2).
	\end{equation}
	Next, we estimate
	\begin{eqnarray*}
	\sup_{x\in \partial \Gamma_\varepsilon}\left|\varepsilon\int_{\partial \Gamma_\varepsilon}Q_{\varepsilon}(x,y)\partial_{\nu}u_{\varepsilon}(y)dvol_h(y)\right| 
	&\leq& \sup_{x\in \partial \Gamma_\varepsilon}\left|\varepsilon \partial_{\nu}u_{\varepsilon}\right|
	\sup_{x\in \partial \Gamma_\varepsilon} \int_{\partial \Gamma_\varepsilon}\left|Q_{\varepsilon}(x,y)\right|dvol_h(y)\\
&\leq& C \varepsilon^2\sup_{x\in \partial \Gamma_\varepsilon}\left| \partial_{\nu}u_{\varepsilon}\right|
	\end{eqnarray*}
The last estimate comes from Lemma \ref{sing str of E}. Combine this estimate with \eqref{epsilonpartialnu} we obtain that
	\begin{equation*}
	\varepsilon \partial_{\nu_x}u_{\varepsilon}(x) = - \frac{|M_{\varepsilon}|}{2\pi} + O_{L^{\infty}(\partial \Gamma_\varepsilon)}(\varepsilon).
	\end{equation*}
	This completes the proof.
\end{proof}

\begin{proof}[Proof of Theorem \ref{main}]
	We first prove ii) then proceed with i). By Proposition \ref{partial_u}, we can express
	\begin{equation*}
	\partial_{\nu}\left.u_{\varepsilon}\right|_{\partial \Gamma_\varepsilon} = -\frac{|M_{\varepsilon}|}{2\pi\varepsilon} + W_{\varepsilon}
	\end{equation*}
\begin{eqnarray}
\label{W bound}
\|W_\varepsilon\|_{L^\infty(\partial\Gamma_\epsilon)} \leq C.
\end{eqnarray}
uniformly in $\epsilon>0$. Then, for $x\in M_{\varepsilon}\setminus\partial \Gamma_\varepsilon$, \eqref{u_partialu_eq} gives
	\begin{multline*}
	\frac{1}{|M|}\int_{M_{\varepsilon}}u_{\varepsilon}(y)dvol_g(y) - u_{\varepsilon}(x) - \frac{|M_{\varepsilon}|}{2\pi\varepsilon}\int_{\partial \Gamma_\varepsilon}E(x,y)dvol_h(y) + \int_{\partial \Gamma_\varepsilon}E(x,y)W_{\varepsilon}(y)dvol_h(y)\\
	= I_{\varepsilon}(x_0,x),
	\end{multline*}
	or equivalently 
	\begin{multline}\label{before restriction}
		\frac{1}{|M|}\int_{M_{\varepsilon}}u_{\varepsilon}(y)dvol_g(y) = u_{\varepsilon}(x) + \frac{|M_{\varepsilon}|}{2\pi\varepsilon}\int_{\partial \Gamma_\varepsilon}E(x,x_0)dvol_h(y)\\
		+ \frac{|M_{\varepsilon}|}{2\pi\varepsilon}\int_{\partial \Gamma_\varepsilon}(E(x,y) - E(x,x_0))dvol_h(y)
		-\int_{\partial \Gamma_\varepsilon}E(x,y)W_{\varepsilon}(y)dvol_h(y) + I_{\varepsilon}(x,x_0).
	\end{multline}
	To compute the left-hand side, we restrict this to $\partial \Gamma_\varepsilon$ where $u_{\varepsilon} = 0$. We note that Proposition \ref{sing structure E} combined with Lemma \ref{dstar} shows that in the coordinates \eqref{geod coord} the leading  singuarlity of the Green's function $E(x(s,x_0,), x(t, x_0))$ is of the form
$$ E(x(s,x_0), x(t,x_0)) =C\log|t-s| + L^\infty(\D_\rho\times \D_\rho)$$
Combine this with Lemma \ref{I} and \eqref{W bound} gives
	\begin{equation*}
	\sup_{x\in\partial \Gamma_\varepsilon}\left|\int_{\partial \Gamma_\varepsilon}E(x,y)W_{\varepsilon}(y)dvol_h(y)\right| = O(\varepsilon\log\varepsilon), \qquad \sup_{x\in\partial \Gamma_\varepsilon}\left|I_{\varepsilon}(x_0,x)\right| = O(\varepsilon^2\log\varepsilon)
	\end{equation*}
	as $\varepsilon\rightarrow 0$. Therefore, restricting \eqref{before restriction} to $\partial \Gamma_\varepsilon$ and using Proposition \ref{sing structure E}, we obtain
	\begin{multline}\label{ave u}
		\frac{1}{|M|}\int_{M_{\varepsilon}}u_{\varepsilon}(y)dvol_g(y) = -\frac{|M_{\varepsilon}||\partial \Gamma_\varepsilon|}{4\pi^2\varepsilon}\log\varepsilon +\frac{|M_{\varepsilon}||\partial \Gamma_\varepsilon|}{2\pi\varepsilon}P_{-4}(x_0,x_0)\\
		+ \frac{|M_{\varepsilon}|}{2\pi\varepsilon}\int_{\partial \Gamma_\varepsilon}(E(x,y) - E(x,x_0))dvol_h(y)\left.\right|_{x\in \partial\Gamma_\varepsilon} + O_{L^\infty(\partial \Gamma_{\varepsilon})}(\varepsilon\log\varepsilon).
	\end{multline}
	Let us examine the third term of the right-hand side
	\begin{multline}\label{difference of E}
		\int_{\partial \Gamma_\varepsilon}(E(x,y) - E(x,x_0))dvol_h(y)\\
		 = \int_{\partial \Gamma_\varepsilon}(\log d_g(x,y) - \log d_g(x,x_0))dvol_h(y) + \int_{\partial \Gamma_\varepsilon}(P_{-4}(x,y) - P_{-4}(x,x_0))dvol_h(y),
	\end{multline}
	where $x\in \partial\Gamma_\varepsilon$. Joint differentiability of $P_{-4}$ gives
	\begin{equation}
\label{estimate the smooth part}
	\sup_{x\in \partial\Gamma_\varepsilon}\left|	\int_{\partial \Gamma_\varepsilon}(P_{-4}(x,y) - P_{-4}(x,x_0))dvol_h(y)\right| \leq C\varepsilon^2.
	\end{equation}
	To investigate the first term of the right-hand side of \eqref{difference of E}, we use the coordinate system $x^{\varepsilon}(\cdot,x_0)$ with $x = x^{\varepsilon}(t,x_0)$ and $y = x^{\varepsilon}(s,x_0)$. Let $d\sigma(s)$ be the pull back of the volume form $dvol_h(y)$ under $s\mapsto x^{\varepsilon}(s,x_0)$, then
	\begin{equation*}
	dvol_h(y) = \varepsilon(1+v_\varepsilon(s))d\sigma(s)
	\end{equation*}
	for some smooth function $\|v_{\varepsilon}\|_{L^\infty(\partial\Gamma_\varepsilon)} \leq C\varepsilon$. By Lemma \ref{d inverse} we have
	\begin{equation*}
		d^{-1}_g(x^\varepsilon(t,x_0),x^\varepsilon(s,x_0)) = \varepsilon^{-1}|t - s|^{-1} + \varepsilon |t - s|^{-1}A(\varepsilon, s,r,\omega)
	\end{equation*}
	for some smooth function $A(\epsilon, s, r,\omega)$ in the variables $(\varepsilon,s,r,\omega)\in [0,\varepsilon_0]\times\mathbb{D}\times\mathbb{R}\times S^1$, where $r=|t-s|$ and $\omega = \frac{t-s}{|t-s|}$. Therefore, the first term of \eqref{difference of E} becomes
	\begin{align}\label{difference of ln}
		\nonumber\int_{\partial \Gamma_\varepsilon}(\log d_g(x,y) - \log d_g(x,x_0))dvol_h(y) 
		&=\varepsilon \int_{\partial\mathbb{D}} \log\left(\frac{1}{\varepsilon^{-1}|t - s|^{-1} + \varepsilon |t - s|^{-1}A}\right)(1+v_\varepsilon(s))d\sigma(s)\\
		&=\varepsilon \int_{\partial\mathbb{D}} \log\left(\frac{\varepsilon |t - s|}{1 + \varepsilon^2A}\right)(1+v_\varepsilon(s))d\sigma(s).
	\end{align}
	We have trivially that 
	\begin{equation}
\label{epslogeps term}
		\varepsilon\left| \int_{\partial\mathbb{D}}\log\left(\frac{\varepsilon}{1 + \varepsilon^2 A(\varepsilon, s, r,\omega)}\right) (1+v_\varepsilon(s))d\sigma(s)\right| \leq C\varepsilon\log\varepsilon
	\end{equation}
	and
	\begin{align}
\label{epsilonsquared estimate}
		\left|\varepsilon \int_{\partial\mathbb{D}}\log|t - s| v_\varepsilon(s)d\sigma(s)\right|& \leq \varepsilon\int_{\partial\mathbb{D}} \left|\log|t - s|v_\varepsilon(s)\right|d\sigma(s) \leq C \varepsilon^2
	\end{align}
for all $t\in \partial\D$. Therefore, inserting the estimates \eqref{epslogeps term} and \eqref{epsilonsquared estimate} into \eqref{difference of ln} gives that for $t\in \partial\D$,
	\begin{align*}
		\int_{\partial \Gamma_\varepsilon}(\log d_g(x,y) - \log d_g(x,x_0))dvol_h(y) &=\varepsilon  \int_{\partial\mathbb{D}}\log|t - s|d\sigma(s) +  O_{L^{\infty}(\partial \Gamma_\varepsilon)}(\varepsilon\log\varepsilon)\\
		&=\frac{\varepsilon}{2}\int_{0}^{2\pi} \log(2-2\cos \theta')d\theta'
		+O_{L^{\infty}(\partial \Gamma_\varepsilon)}(\varepsilon\log\varepsilon)\\
		& = O_{L^{\infty}(\partial \Gamma_\varepsilon)}(\varepsilon\log\varepsilon)
	\end{align*}

This combined with \eqref{estimate the smooth part} and \eqref{difference of E} gives 
$$\left|\int_{\partial \Gamma_\varepsilon}(E(x,y) - E(x,x_0))dvol_h(y)\right| \leq C\varepsilon \log\varepsilon.$$
Inserting this estimate into \eqref{ave u} implies
	\begin{multline*}
	\frac{1}{|M|}\int_{M_{\varepsilon}}u_{\varepsilon}(y)dvol_g(y) = -\frac{|M_{\varepsilon}||\partial \Gamma_\varepsilon|}{4\pi^2\varepsilon}\log\varepsilon +\frac{|M_{\varepsilon}||\partial \Gamma_\varepsilon|}{2\pi\varepsilon}P_{-4}(x_0,x_0)\\
	+ O_{L^{\infty}(\partial \Gamma_\varepsilon)}(\varepsilon\log\varepsilon).
	\end{multline*}
	We take the supremum norm over $\partial \Gamma_\varepsilon$ to obtain
	\begin{multline}\label{av}
	\frac{1}{|M|}\int_{M_{\varepsilon}}u_{\varepsilon}(y)dvol_g(y) = -\frac{|M_{\varepsilon}||\partial \Gamma_\varepsilon|}{4\pi^2\varepsilon}\log\varepsilon +\frac{|M_{\varepsilon}||\partial \Gamma_\varepsilon|}{2\pi\varepsilon}P_{-4}(x_0,x_0)\\
	+  O(\varepsilon\log\varepsilon).
	\end{multline}
	Next, we note that $|\Gamma_\varepsilon| = |B_\varepsilon(x_0)| = O(\varepsilon^2)$ and
	\begin{equation*}
		|\partial \Gamma_\varepsilon|=|\partial B_\varepsilon(x_0)| = \int_{\partial B_\varepsilon(x_0)}dvol_h(y) = \int_{\partial \mathbb{D}}\varepsilon (1 + v_{\varepsilon}(s))d\sigma(s) = 2\pi \varepsilon +O (\varepsilon^2).
	\end{equation*}
	This gives us part ii) of Theorem \ref{main}. 
	
	Let us put \eqref{av} into \eqref{u_partialu_eq} to we obtain
	\begin{multline*}
	u_{\varepsilon}(x) = -\frac{|M|}{2\pi}\log\varepsilon + |M_{\varepsilon}|P_{-4}(x_0,x_0) \\
	+\int_{\partial M_{\varepsilon}}E(x,y)\partial_{\nu}u_{\varepsilon}(y)dvol_h(y) - I_{\varepsilon}(x_0,x) + O(\varepsilon\log\varepsilon).
	\end{multline*}
	as $\varepsilon \rightarrow 0$. It remains to show that for a fixed compact set $K\subset\subset M_\varepsilon$,
	\begin{equation}
\label{volume term}
	\int_{\partial \Gamma_\varepsilon}E(x,y)\partial_{\nu}u_{\varepsilon}(y)dvol_h(y) - I_{\varepsilon}(x_0,x) = - |M| E(x,x_0) + r_{\varepsilon}(x).
	\end{equation}
for some $r_\varepsilon(x)$ whose $L^\infty(K)$ norm is of order $\varepsilon$.

To do this, let us fix any compact set $K\subset M$ which does not contain $x_0$, so that $K\cap \Gamma_\varepsilon$ is empty for sufficiently small $\varepsilon>0$. Then $E(\cdot,\cdot)$ is smooth in $K\times \Gamma_\varepsilon$, and hence we estimate
	\begin{equation*}
	\sup_{x\in K} I(x_0,x) = \sup_{x\in K} \int_{\Gamma_\varepsilon} E(x,y) dvol_g(y) = O(\varepsilon^2) \qquad \text{as } \varepsilon \rightarrow 0.
	\end{equation*}
	Next, we write 
	\begin{multline*}
	\int_{\partial \Gamma_\varepsilon}E(x,y)\partial_{\nu}u_{\varepsilon}(y)dvol_h(y) = \int_{\partial \Gamma_\varepsilon}(E(x,y) - E(x,x_0))\partial_{\nu}u_{\varepsilon}(y)dvol_h(y)\\
	+E(x,x_0) \int_{\partial \Gamma_\varepsilon}\partial_{\nu}u_{\varepsilon}(y)dvol_h(y). 
	\end{multline*}
The first integral can be estimated using Proposition \ref{partial_u} and the smoothness of $E(\cdot,\cdot)$ in $K\times \Gamma_\varepsilon$to give

$$ \int_{\partial \Gamma_\varepsilon}E(x,y)\partial_{\nu}u_{\varepsilon}(y)dvol_h(y) = O_{L^\infty(K)}(\varepsilon) +  E(x,x_0) \int_{\partial \Gamma_\varepsilon}\partial_{\nu}u_{\varepsilon}(y)dvol_h(y). $$
Now use the compatibility condition \eqref{comp cond} to get for $x\in K$,
	\begin{equation*}
	 \int_{\partial \Gamma_\varepsilon}E(x,y)\partial_{\nu}u_{\varepsilon}(y)dvol_h(y) = - |M_\varepsilon| E(x,x_0) + O_{L^\infty(K)}(\varepsilon) = - |M| E(x,x_0) + O_{L^\infty(K)}(\varepsilon).
	\end{equation*}
	This gives Part i) of Theorem \ref{main}.
	
\end{proof}

\section{Narrow capture problem on the surface with boundary}\label{proof of thm with bry}
Here, we consider the same problem for the case when the surface has a smooth boundary , $\partial M$, which reflects the particle. Without loss of generality, we assume that $M$ is an connected open subset of a compact orientable Riemannian manifold $(\tilde{M}, g)$ without boundary. Let $\tilde{E}(x,y)$ be the Green's function on $\tilde{M}$, given by \eqref{no boundary green}. The Neumann Green's function $E(x,y)$ is given by, for $x\in M^0$,
\begin{equation}\label{Neumann Green function}
	\begin{cases}
		\Delta_{g,y}E(x,y) = -\delta_x(y) + \frac{1}{|M|}, & \text{for } y\in M,\\
		\partial_{\nu_y}E(x,y)=0, & \text{for } y\in \partial M,\\
		\int_M E(x,y)dvol_g(y) = 0.
	\end{cases}
\end{equation}
We can obtain this function by setting $E = \tilde{E} - C$, where the correction term $C(x,y)$ is the solution to the boundary value problem, for $x\in M^0$,
\begin{equation*}
	\begin{cases}
		\Delta_{g,y}C(x,y) = \frac{1}{|\tilde{M}|} - \frac{1}{|M|}, & \text{for } y\in M,\\
		\partial_{\nu_y}C(x,y) = \partial_{\nu_y}\tilde{E}(x,y), & \text{for } y\in \partial M,\\
		\int_M C(x,y)d_g(y) = \int_M \tilde{E}(x,y)dvol_g(y).
	\end{cases}
\end{equation*}
Therefore, for $U\subset \subset M$ away from the boundary, that is $\mathrm{dist}_g(U, \partial M) > 0$, it follows
\begin{equation*}
	C = \tilde{E} - E \in C^{\infty}(\bar{U} \times \bar{U}).
\end{equation*}
Hence, we can decompose
\begin{equation}\label{dec2}
E(x,y) = -\frac{1}{2\pi} \log d_g(x,y) + P_{-4}(x,y),
\end{equation}
where $P_{-4}(x,y) \in C^1(\bar{U}\times \bar{U})$. Moreover, since we are considering the centre of the trap, $x_0$, to be fixed in the interior of $M$ with a sufficiently small $\varepsilon >0$, we have that 
\begin{equation*}
	\sup_{x\in \partial \Gamma_\varepsilon} \int_{\Gamma_\varepsilon} C(x, y) dvol_g(y) = O(\varepsilon^2),
	\qquad
	\sup_{x\in \partial \Gamma_\varepsilon} \partial_{\nu_x} \int_{\Gamma_\varepsilon} C(x, y) dvol_g(y) = O(\varepsilon^2),
\end{equation*}
\begin{equation*}
	\sup_{x\in \partial \Gamma_\varepsilon} \partial_{\nu_x} \int_{\partial \Gamma_\varepsilon} C(x, y) dvol_h(y) = O(\varepsilon),
\end{equation*}
as $\varepsilon\rightarrow 0$. As a result, we obtain the following analogues of Lemmas \ref{I}-\ref{sing str of E} for the function
\begin{equation*}
	I_\varepsilon(x_0,x) := \int_{\Gamma_\varepsilon}E(x, y) dvol_g(y),
\end{equation*}
\begin{lemma}\label{analogue}
	As $\varepsilon\rightarrow 0$, we have that
	\begin{equation*}
		\sup_{x\in \partial \Gamma_\varepsilon} I_{\varepsilon}(x_0, x) = O(\varepsilon^2 \log\varepsilon),
		\qquad
		\sup_{x\in \partial \Gamma_\varepsilon} \partial_{\nu_x} I(x_0,x) = O(\varepsilon),
	\end{equation*}
	\begin{equation*}
		\partial_{\nu_x} E(x,y)\left.\right|_{x,y\in \partial \Gamma_\varepsilon} = \frac{1}{4\pi\varepsilon} + Q_{\varepsilon}(x,y),
	\end{equation*}
	for some function $Q_\varepsilon$ such that 
	\begin{equation*}
	\sup_{x\in \partial \Gamma_\varepsilon} \int_{\partial \Gamma_\varepsilon}  Q_{\varepsilon}(x,y) dvol_h(y) = O(\varepsilon).
	\end{equation*}
\end{lemma}

Note that, in this case, $\partial M \neq \emptyset$, the mean first-passage time $\mathbb{E}[\tau_{\Gamma_{\varepsilon}} | X_0 = x]$ satisfies the following mixed boundary value problem, see Appendix in \cite{NTT2021},
\begin{equation}\label{bvp2}
	\begin{cases}
		\Delta_g u_\varepsilon = 1 & \text{on } M_\varepsilon,\\
		u_\varepsilon = 0, & \text{on } \partial \Gamma_\varepsilon,\\
		\partial_{\nu} u = 0, & \text{on } \partial M.
	\end{cases}
\end{equation}
which gives the compatibility condition
\begin{equation}\label{comp cond 2}
\int_{\partial \Gamma_\varepsilon}\partial_{\nu}u_{\varepsilon}(y)dvol_h(y) = - |M_{\varepsilon}|.
\end{equation}
As a result, we have the following analogue of Proposition \ref{partial_u}
\begin{proposition}\label{partial_u_2}
	Let $u_{\varepsilon}$ be the solution of \eqref{bvp2}, then
	\begin{equation*}
	\partial_{\nu}\left.u_{\varepsilon}\right|_{\partial \Gamma_\varepsilon} = -\frac{|M_{\varepsilon}|}{2\pi\varepsilon} + W_\varepsilon.
	\end{equation*}
	for some $W_\varepsilon\in O_{L^{\infty}(\partial \Gamma_\varepsilon)}(1)$ as $\varepsilon\rightarrow 0$.
\end{proposition}
\begin{proof}
	By the Green's identity,
	\begin{multline*}
	\frac{1}{|M|}\int_{M_{\varepsilon}}u_{\varepsilon}(y)dvol_g(y) - u_{\varepsilon}(x) - I_{\varepsilon}(x_0,x)\\
	= \int_{\partial M_{\varepsilon}}\partial_{\nu_y}E(x,y)u_{\varepsilon}(y)dvol_h(y) - \int_{\partial M_{\varepsilon}}E(x,y)\partial_{\nu}u_{\varepsilon}(y)dvol_h(y).
	\end{multline*}
	Using the boundary condition for both $E$ and $u_\varepsilon$, we obtain
	\begin{equation*}
	\frac{1}{|M|}\int_{M_{\varepsilon}}u_{\varepsilon}(y)dvol_g(y) - u_{\varepsilon}(x) + \int_{\partial \Gamma_\varepsilon}E(x,y)\partial_{\nu}u_{\varepsilon}(y)dvol_h(y) = I_{\varepsilon}(x_0,x).
	\end{equation*}
	Next, as in Proposition \ref{partial_u}, we take $\partial_{\nu_x}$, restrict to $\partial M_\varepsilon$, and use Lemma \ref{analogue} to obtain
	\begin{equation*}
	- \partial_{\nu}u_{\varepsilon}(x) + \partial_{\nu} \int_{\partial \Gamma_\varepsilon}E(x,y)\partial_{\nu}u_{\varepsilon}(y)dvol_h(y) = O_{L^{\infty}(\partial \Gamma_\varepsilon)}(\varepsilon).
	\end{equation*}
	By Proposition 11.3 of \cite{taylor2},
	\begin{equation*}
	- \frac{1}{2} \partial_{\nu}u_{\varepsilon}(x) + \int_{\partial \Gamma_\varepsilon}\partial_{\nu_x}\tilde{E}(x,y)\partial_{\nu}u_{\varepsilon}(y)dvol_h(y) - \int_{\partial \Gamma_\varepsilon}\partial_{\nu_x}C(x,y)\partial_{\nu}u_{\varepsilon}(y)dvol_h(y) = O_{L^{\infty}(\partial \Gamma_\varepsilon)}(\varepsilon)
	\end{equation*}
	so that
	\begin{equation*}
	- \frac{1}{2} \partial_{\nu}u_{\varepsilon}(x) +  \int_{\partial \Gamma_\varepsilon}\partial_{\nu_x}E(x,y)\partial_{\nu}u_{\varepsilon}(y)dvol_h(y) = O_{L^{\infty}(\partial \Gamma_\varepsilon)}(\varepsilon).
	\end{equation*}
	Further, we repeat the steps of the proof of Proposition \ref{partial_u} with replacing Lemma \ref{sing str of E} to Lemma \ref{analogue}.
\end{proof}
Repeating the proof of Theorem \ref{main} by replacing Proposition \ref{partial_u} with \ref{partial_u_2} and Lemmas \ref{I}-\ref{sing str of E} with \ref{analogue} yields the following theorem
\begin{theorem}\label{main2}
	Let $(M,g,\partial M)$ be a compact, connected and orientable Riemannian surface with smooth boundary. Fix $x_0\in M^0$ and let $\Gamma_{\varepsilon}:=B_\varepsilon(x_0)$ be a geodesic ball centered at $x_0$ of geodesic radius $\varepsilon >0$ such that $\partial \Gamma_\varepsilon \cap \partial M = \emptyset$ .\\
	\\
	i) For each $x\notin \Gamma_\varepsilon$, the first-passage time satisfies the  following asymptotic formula, as $\varepsilon\rightarrow 0$,
	\begin{multline*}
	\mathbb{E}[\tau_{\Gamma_{\varepsilon}} | X_0 = x] = -\frac{|M|}{2\pi}\log\varepsilon + |M|P_{-4}(x_0,x_0)
	- |M| E(x,x_0) + r_{\varepsilon}(x) + O(\varepsilon\log\varepsilon).
	\end{multline*}
	for some function $r_\varepsilon$ such that $\|r_\varepsilon\|_{C(K)} \leq C_{K} \varepsilon$ for any compact $K\subset M$ for which $K\cap \Gamma_\varepsilon = \emptyset$. The Neumann Green's function $E(x,y)$ is given by \eqref{Neumann Green function} and $P_{-4}(x_0,x_0)$ is the evaluation at $(x,y)=(x_0,x_0)$ of the kernel $P_{-4}(x,y)$ in \eqref{dec2}.\\
	\\
	ii) Let $M_\varepsilon=M\setminus \Gamma_{\varepsilon}$, then the spatial average of the mean first-passage time satisfies the asymptotic formula, as $\varepsilon\rightarrow 0$,
	\begin{equation*}
	\frac{1}{|M|}\int_{M_{\varepsilon}}\mathbb{E}[\tau_{\Gamma_{\varepsilon}} | X_0 = y]dvol_g(y) = -\frac{|M|}{2\pi}\log\varepsilon +|M|P_{-4}(x_0,x_0)+  O(\varepsilon\log\varepsilon).
	\end{equation*}
\end{theorem}

\begin{appendix}
\section{Proof of Proposition \ref{sing structure E}}
\label{proof of lemma 3.1}
In this section we provide a brief outline for the necessary aspects of the theory of pseudo-differential operators. For a greater in-depth description of $\mathit{\Psi DO}$, we refer to the reader to \cite{Hor3}, \cite{taylor2} or \cite{doi:10.1142/4047}. After the basic elements of $\mathit{\Psi DO}$s have been described, we offer a proof for Proposition \ref{sing structure E}. 
\subsection{Overview of Pseudo-differential operators ($\mathit{\Psi DO}$) on Manifolds}
Let $p(x,\xi) \in C^\infty (T^*\mathbb{R}^n)$. We call $p(x,\xi)$ a standard symbol of order $m$ if for all $m\in \mathbb{R}$, the following estimate holds uniformly 
\begin{align*}
    |D^\alpha_x D^\beta_\xi p(x,\xi)| \lesssim \langle\xi\rangle^{m-|\beta|}
\end{align*}
for every multi-index $\alpha,\beta\in \mathbb{N}^n$. Should $p(x,\xi)$ be an order $m$ standard symbol, we say that $p(x,\xi) \in S^m_{1,0}(T^*\mathbb{R}^n)$. 
\begin{Rem}
    We use $D_\xi := -i \partial_\xi, D_x := -i\partial_x$ and $\langle\xi\rangle:= (1+|\xi|^2)^{1/2}$. 
\end{Rem}
Of particular interest is a subspace of $S^m_{1,0}(T^*\mathbb{R}^n)$ known as the \emph{classical symbols of order $m$} denoted by $S^m_{cl}(T^*\mathbb{R}^n)$. Such symbols are defined via a homogeneity requirement on the asymptotic expansion of $p(x,\xi)$ 
\begin{equation}
    p(x,\xi) \sim \sum_{j=0}^\infty p_{m-j}(x,\xi). \label{homogeneous expansion}
\end{equation}
where $p_{m-j}(x,\xi)$ are homogeneous of order $m-j$ in the fiber for all $x\in \mathbb{R}^n$. i.e. $p_{m-j}(x,\lambda \xi) = \lambda^{m-j}p_{m-j}(x,\xi)$ for $\lambda, |\xi|\geq 1$. The above expansion \eqref{homogeneous expansion} is an asymptotic expansion in the sense that 
\begin{align*}
    p(x,\xi) - \sum_{j=0}^N p_{m-j}(x,\xi) \in S^{m-N-1}_{1,0}(T^*\mathbb{R}^n).
\end{align*}
If $p(x,\xi)\in S^m_{1,0}(T^*\mathbb{R}^n)$, we can define an operator $p(x,D):C_c^\infty(\mathbb{R}^n) \rightarrow \mathcal{D}'(\mathbb{R}^n)$ which is given locally by the following expression
\begin{equation}
    p(x,D)v:=\int_{\mathbb{R}^n} e^{i\xi\cdot x} p(x,\xi) \widehat{v}(\xi) d\xi. \label{local psido rep}
\end{equation}
Such an operator is called an $m$-th order pseudo-differential operator and we say that $p(x,D) \in \Psi^m_{1,0}(\mathbb{R}^n)$. We can also define $\Psi^m_{cl}(\mathbb{R}^n)$ by requiring $p(x,\xi)\in S^m_{cl}(T^*\mathbb{R}^n)$ in \eqref{local psido rep}. Furthermore, we can uniquely extend $p(x,D)$ to a bounded linear operator $p(x,D):H^k(\mathbb{R}^n) \rightarrow H^{k-m}(\mathbb{R}^n)$ for $k\in \mathbb{R}$. We also define the space of \emph{smoothing operators}, pseudo-differential operators with smooth kernels \emph{along the diagonal} as 
\begin{align*}
    \Psi^{-\infty}(\mathbb{R}^n) := \bigcup_{m\in \mathbb{R}} \Psi^m(\mathbb{R}^n).
\end{align*}
Smoothing operators arise as pseudo-differential operators of symbols belonging to the space defined by 
\begin{align*}
    S^{-\infty}(T^*\mathbb{R}^n) = \bigcup_{m\in \mathbb{R}} S^m(T^*\mathbb{R}^n).
\end{align*}
We also have that if $p(x,D)\in \Psi^m(\mathbb{R}^n)$ and $q(x,D)\in \Psi^l(\mathbb{R}^n)$, then $p(x,D)q(x,D) \in \Psi^{m+l}(\mathbb{R}^n)$. The way such composition is defined is via a \emph{symbol calculus}. The symbol for $p(x,D)q(x,D)$, denoted by $(p\#q)(x,\xi)$ is given by 
\begin{equation}
    (p\#q)(x,\xi) \sim \sum_\mu \frac{i^{|\mu|}}{\mu!} D_\xi^\mu p(x,\xi)D_x^\mu q(x,\xi). \label{product formula}
\end{equation}
where $\mu\in \mathbb{N}^n$ denotes a multi-index. The derivation for this formula can be found in \cite{taylor2}, Chapter 7, Section 3. Another important aspect of pseudo-differential operators, which will largely be used in the proof of proposition 3.1 is the notion of elliptic parametrices. First, if $p(x,D)\in \Psi^m(\mathbb{R}^n)$, we say that $p(x,D)$ is elliptic if the following lower bound estimate holds for constants $C,R>0$
\begin{align*}
    |p(x,\xi)| \geq C(1+|\xi|)^m, \;\ \text{for} |\xi| \geq R.
\end{align*}
If $p(x,D)$ is elliptic, then the following theorem holds 
\begin{theorem}
    If $p(x,D)\in \Psi^m(\mathbb{R}^n)$ is elliptic, then there exists a $q(x,D),\tilde{q}(x,D)\in \Psi^{-m}(\mathbb{R}^n)$ such that 
    \begin{align*}
        p(x,D)q(x,D) &= I + \Psi^{-\infty}(\mathbb{R}^n) \\
        \tilde{q}(x,D)p(x,D) &= I + \Psi^{-\infty}(\mathbb{R}^n)
    \end{align*}
\end{theorem}
The proof for the above theorem can be found in \cite{taylor2}, Chapter 7, Section 4. The operators $q(x,D),\tilde{q}(x,D)$ are known as right and left parametrices of $p(x,D)$ respectively. Furthermore, it is a straightforward corollary that $q(x,D) = \tilde{q}(x,D)+\Psi^{-\infty}(\mathbb{R}^n)$.
\\~\\
In lieu of the pseudo-differential theory on $\mathbb{R}^n$, there is a natural extension to $C^\infty$-manifolds. Let $M$ be a closed manifold. An operator $A:C^\infty(M)\rightarrow \mathcal{D}'(M)$ is said to belong to $\Psi^m_{1,0}(M)$ if there is an atlas $(U_j,\varphi_j)$ covering $M$, with $\varphi_j :U_j \rightarrow V_j\subset \mathbb{R}^n$ and a partition of unity $\{\chi_j\}$ subordinate to the atlas covering such that the following operator
\begin{align*}
    u\mapsto (\chi_kA\chi_j\varphi_j^*u)\circ \varphi_k^{-1}.
\end{align*}
belongs to $\Psi^m_{1,0}(\mathbb{R}^n)$. Similarly, if $a\in C^\infty(T^*M)$, we say that $a\in S^m_{1,0}(T^*M)$ if 
\begin{align*}
    \chi_j\circ \varphi_j^{-1}a(\varphi_j^{-1}(\cdot),\varphi_j^*\cdot) \in S^m_{1,0}(T^*\mathbb{R}^n).
\end{align*}
The classical pseudo-differential operators and symbols on $M$ are defined in the same way. 
\subsection{Proof of Proposition 3.1}
Since $\Delta_g\in \Psi^2_{cl}(M)$ elliptic, we have that as a result of Theorem A.2, there is a parametrix $P\in \Psi^{-2}_{cl}(M)$ satisfying the following equation 
\begin{align*}
    \Delta_gP = I +\Psi^{-\infty}(M)
\end{align*}
Furthermore, as a corollary of Borel's lemma ,\cite{doi:10.1142/4047}, we can express the Schwartz kernel of $P$ as 
\begin{align*}
    P(x,y) = \sum_{j=0}^\infty P_{-2-j}(x,y) 
\end{align*}
where $P_{-2-j}\in \Psi^{-2-j}_{cl}(M)$. A standard first order parametrix construction indicates that we can choose for $x$ near $y$
\begin{align*}
    P_{-2}(x,y) = -\frac{1}{2\pi}\log d_g(x,y) 
\end{align*}
So, our claim is that $P_{-3} = 0$. This problem, reduces to showing that 
\begin{equation}
    P-P_{-2} \in \Psi^{-4}_{cl}(M) \label{Reduction 1}
\end{equation}
Left composition of $\Delta_g$ with \eqref{Reduction 1} results in the following equivalent formulation
\begin{equation}
    \Delta_g P_{-2} -I \in \Psi^{-2}_{cl}(M) \label{Reduction 2}
\end{equation}
Self-adjointness of $P_{-2}$ and $\Delta_g$imply that \eqref{Reduction 2} is equivalent to 
\begin{equation}
    P_{-2}\Delta_g - I \in \Psi^{-2}_{cl}(M) \label{Adjoint}
\end{equation}
Should \eqref{Adjoint} be true, then we would infer that the expansion for $P_{-2}\Delta_g$ consists of no $-1$ order pseudo-differential operator. This is equivalent to requiring the principle symbol, which is homogeneous of degree $-1$ satisfy the following 
\begin{align*}
    \sigma_{-1}(P_{-2}\Delta_g-I)(y,\eta) = 0, \,\ \text{for all} \,\ (y,\eta)\in T^*M
\end{align*}
In order to attain the above requirement, we show that $\sigma_{-1}(P_{-2}\Delta_g-I)$ can be bounded from above in the following manner
\begin{equation}
    |\sigma_{-1}(P_{-2}\Delta_g-I)(y_0,\tau \eta_0)| \lesssim \tau^{-2} \label{Estimate for principle symbol}
\end{equation}
for $\tau \rightarrow \infty$ and fixed $(y_0,\eta_0)\in S^*M$. Since the decay is radially symmetric and is independent of the choice of $y_0$, \eqref{Estimate for principle symbol} implies that $\sigma_{-1}(P_{-2}\Delta_g-I)$ vanishes on $T^*M$. Now, we let $\Phi:V \rightarrow U$ be a Riemann normal co-ordinate chart, centered at $y_0$ for which $\Phi(0) = y_0\in U\subset M$. Let $A:C_c^\infty(\mathbb{R}^2)\rightarrow \mathcal{D}'(\mathbb{R}^2)$ and $B:C_c^\infty(\mathbb{R}^2)\rightarrow \mathcal{D}'(\mathbb{R}^2)$ denote the pull-back operators for $P_{-2}$ and $\Delta_g$ by $\Phi$ respectively. Then, by the invariance of principle symbols under symplectomorphism, we have that
\begin{align*}
    \sigma_{-1}(P_{-2}\Delta_g - I)(y_0,\eta_0) = \sigma_{-1}(AB-I)(0,\xi) 
\end{align*}
If $a(t,\xi)$ and $b(t,\xi)$ denote the symbols of $A$ and $B$, then by \eqref{product formula}, we have that 
\begin{align*}
    (a\#b)(t,\xi) = a(t,\xi)b(t,\xi) -i\sum_{|\mu|=1} D^\mu_\xi a(t,\xi) D^\mu_t b(t,\xi) + S^{-2}_{cl}(T^*\mathbb{R}^2)
\end{align*}
Furthermore, if we restrict $t=0$, since we are working in Riemannian normal co-ordinates, we have that $\left.D^\mu_t b(t,\xi)\right\vert_{t=0}=0$, which is shown in \cite{lee}, and thus
\begin{align*}
    (a\#b)(0,\xi) = a(0,\xi)b(0,\xi) + S^{-2}_{cl}(T^*\mathbb{R}^2)
\end{align*}
Since the symbol $a(0,\xi)$ is given by the Schwartz kernel of $A$, where 
\begin{align*}
    a(0,\xi) = -\frac{1}{2\pi} \int_{\mathbb{R}^2} e^{-i\xi\cdot t} \log|t| dt = |\xi|^{-2} 
\end{align*}
we have that 
\begin{align*}
    (a\#b)(0,\xi) = 1 + S^{-2}_{cl}(T^*\mathbb{R}^2)
\end{align*}
This implies that 
\begin{align*}
    |(a\#b)(0,\xi)-1| \lesssim \langle\xi\rangle^{-2}\implies \sigma_{-1}(AB-I)(0,\xi) = 0
\end{align*}
The last equality thus implies that $\sigma_{-1}(P_{-2}\Delta_g-I)(y,\eta) = 0$ for all $(y,\eta)\in T^*M$. 
\end{appendix}

\bibliographystyle{plain}
\bibliography{references}

\setlength{\parskip}{0pt}

\end{document}